\documentclass[runningheads,11pt]{llncs}

\usepackage[utf8]{inputenc}

\usepackage{hyperref}
\usepackage{graphicx}
\usepackage{pgf,tikz}
\usepackage{amssymb,latexsym}
\usepackage{amsmath}
\usepackage{url}
\usepackage{nicefrac}

\usetikzlibrary{trees,decorations}
\usetikzlibrary{arrows,automata}

\let\L\relax
\DeclareMathOperator{\L}{\mathcal{L}}
\DeclareMathOperator{\tr}{\mathrm{tr}}
\DeclareMathOperator{\Ei}{\mathrm{Ei}}

\usepackage[margin=1.0in]{geometry}

\usepackage[footnotesize]{caption}

\newcommand{\beq}{\begin{equation}}
\newcommand{\eeq}{\end{equation}}
\newcommand{\bea}{\begin{array}}
\newcommand{\eea}{\end{array}}

\title{Weighted de Bruijn Graphs for the Menage Problem and Its Generalizations\thanks{The work is supported by the National Science Foundation under grant No. IIS-1462107.}}
\titlerunning{Menage Problem and Its Generalizations}
\author{Max A. Alekseyev}
\institute{The George Washington University, Washington, DC, USA\\\href{mailto:maxal@gwu.edu}{\texttt{maxal@gwu.edu}}}
\date{}

\begin{document}

\maketitle
\thispagestyle{plain}

\begin{abstract}
We address the problem of enumeration of seating arrangements of married couples around a circular table such that no spouses sit next to each other
and no $k$ consecutive persons are of the same gender.
While the case of $k=2$ corresponds to the classical \emph{probl\`eme des m\'enages} with a well-studied solution, 
no closed-form expression for the number of seating arrangements is known when $k\geq 3$.

We propose a novel approach for this type of problems based on enumeration of walks in certain algebraically weighted de Bruijn graphs.
Our approach leads to new expressions for the menage numbers and their exponential generating function 
and allows one to efficiently compute the number of seating arrangements in general cases, which we illustrate in detail for the ternary case of $k=3$.
\end{abstract}

\section{Introduction}

The famous \emph{menage problem} (\emph{probl\`eme des m\'enages}) asks for the number $M_n$ of seating arrangements of $n$ married couples of opposite sex around a circular table 
such that
\begin{enumerate}
\item no spouses sit next to each other;
\item females and males alternate.
\end{enumerate}
The problem was formulated by Edouard Lucas in 1891~\cite{Lucas1891}. A complete solution was first obtained by Touchard in 1934~\cite{Touchard1934}.

Let us call a couple seating next to each other \emph{close}.
The restriction of the menage problem can be equivalently stated as
\begin{enumerate}
\item there are no close couples;
\item no $k=2$ consecutive people are of the same sex.
\end{enumerate}
This reformulation allows us to generalize the menage problem to other values of $k$, such as $k=3$ which we refer to as the \emph{ternary menage problem}.
The ternary menage problem was posed by Hugo Pfoertner in 2006 as the sequence \texttt{A114939} in the OEIS~\cite{OEIS}, for which he then managed to correctly compute only the first three terms.

In this work, we propose a novel approach for the generalized menage problem based on the transfer-matrix method~\cite[Section 4.7]{Stanley1997} applied to certain algebraically weighted de Bruijn graphs.
We illustrate our approach on the classical case $k=2$, where we obtain new formulae for the menage numbers $M_n$ and their exponential generating function (EGF).
While an explicit expression (in terms of the modified Bessel functions) for the EGF was earlier derived by Wyman and Moser~\cite{Wyman1958},
they admitted it be ``quite complicated''. In contrast, our expression (and its derivation) is much simpler and can be stated in terms of a certain power series or the exponential integral function.
We further apply our approach for the ternary case $k=3$, which apparently has not been addressed in the literature before.
While the resulting formulae in this case are not that simple, they provide an efficient method for computing the corresponding number of seating arrangements,
which we used to compute many new terms for \texttt{A114939} and related sequences in the OEIS.

\section{Classical Approaches for Menage Problem}\label{Sec:AppMen}

To the best of our knowledge, there exist three major approaches for solving the menage problem, which we briefly discuss below.

\paragraph*{Ladies First.} 
A straightforward approach to the menage problem is first to seat all ladies (in $2\cdot n!$ ways) and then to seat all gentlemen, obeying the close couple restriction. 
This way the problem reduces to enumerating placements of 
non-attacking rooks on a square board like the one shown in Fig.~\ref{Fig:Board}.
Using the rook theory~\cite[Section 2.3]{Stanley1997}, this leads to the \emph{Touchard formula}:
\beq\label{eq:Touchard}
M_n = 2 \cdot n!\cdot \sum_{k=0}^n (-1)^k \frac{2n}{2n-k} {2n-k\choose k} (n-k)!\,.
\eeq

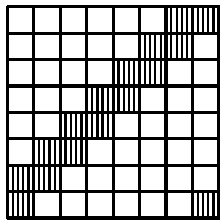
\begin{figure}[t]
\begin{center}
\begin{picture}(80,80)
\thicklines
\multiput(0,0)(10,0){9}{\line(0,1){80}}
\multiput(0,0)(0,10){9}{\line(1,0){80}}
\thinlines
\multiput(0,0)(2,0){5}{\line(0,1){20}}
\multiput(10,10)(2,0){5}{\line(0,1){20}}
\multiput(20,20)(2,0){5}{\line(0,1){20}}
\multiput(30,30)(2,0){5}{\line(0,1){20}}
\multiput(40,40)(2,0){5}{\line(0,1){20}}
\multiput(50,50)(2,0){5}{\line(0,1){20}}
\multiput(60,60)(2,0){5}{\line(0,1){20}}
\multiput(70,70)(2,0){5}{\line(0,1){10}}
\multiput(70,00)(2,0){5}{\line(0,1){10}}
\end{picture}
\end{center}
\caption{A board corresponding to the menage problem with $n=8$ couples. 
For a fixed seating arrangement of ladies, the seating arrangements of gentlemen are in one-to-one correspondence with the placements of $n$ non-attacking rooks at non-shaded cells.}
\label{Fig:Board}
\end{figure}

\paragraph*{Hamiltonian Cycles in Crown Graphs.}
The seating arrangements satisfying the menage problem 
correspond to directed Hamiltonian cycles in the \emph{crown graph} on $2n$ vertices
obtained from the complete bipartite graph $K_{n,n}$ with removal of a perfect matching. 
Here males/females represent the partite sets of $K_{n,n}$ with every male vertex connected to every female vertex, except for the spouses (corresponding to 
the removed perfect matching).
For odd integers $n$, crown graphs on $2n$ vertices represent circulant graphs, where Hamiltonian cycles can be systematically enumerated~\cite{Golin2005}.

\paragraph*{Non-Sexist Inclusion-Exclusion.}
Bogart and Doyle~\cite{Bogart1986} suggested to compute $M_n$ with the inclusion-exclusion principle (e.g., see \cite[Section 2.1]{Stanley1997}) 
as the number of alternating male-female seating arrangements that have no close couples. 
To do so, they computed the number $W_{n,j}$ of alternating male-female seating arrangements of $n$ couples with $j$ fixed couples being close as
\beq\label{eq:Wnj}
W_{n,j} = 2\cdot \frac{2n}{2n-j}\binom{2n-j}{j}\cdot j!\cdot (n-j)!^2\,,
\eeq
where:
\begin{itemize}
\item the factor $2$ accounts for two ways to reserve alternating seats for males and females;
\item $\frac{2n}{2n-j}\binom{2n-j}{j}$ is the number of ways to select $2j$ seats for the $j$ close couples;
\item $j!$ is the number of seating arrangements of the $j$ close couples at the $2j$ selected seats;
\item $(n-j)!^2 = (n-j)!\cdot (n-j)!$ is the number of ways to seat females and males from the other $n-j$ couples.
\end{itemize}
The inclusion-exclusion principle then implies that
\beq\label{eq:Mn1}
\begin{split}
M_n &= \sum_{j=0}^n (-1)^j\cdot \binom{n}{j}\cdot W_{n,j} \\
&= 2\cdot \sum_{j=0}^n (-1)^j\cdot \binom{n}{j}\cdot \frac{2n}{2n-j}\binom{2n-j}{j}\cdot j!\cdot (n-j)!^2\,,
\end{split}
\eeq
which trivially simplifies to \eqref{eq:Touchard}.

The aforementioned approaches for the menage problem 
do not seem to easily extend to the ternary case, 
since there is no nice male-female alternating structure anymore. In particular, the ladies-first approach does not reduce the problem to a uniform board
and there is no obvious reduction to a Hamiltonian cycle problem.
The (non-sexist) inclusion-exclusion approach is most prominent, but it is unclear what should be in place of $\frac{2n}{2n-j}\binom{2n-j}{j}$.
In order to generalize the solution to the menage problem to the ternary case, we suggest to look at this problem at a different angle as described below.

\section{De Bruijn Graph Approach for Menage Problem}

So far, a seating arrangement in the menage problem was viewed as a cyclic (clockwise) sequence of females ($f_i$) and males ($m_j$):
$$f_{i_1} \to m_{j_1} \to f_{i_2} \to m_{j_2} \to \dots \to f_{i_n} \to m_{j_n} \to f_{i_1}\,.$$
However, it can also be viewed as a cyclic sequence of \emph{pairs} of people sitting next to each other:
%\begin{small}
$$(f_{i_1},m_{j_1}) \to (m_{j_1},f_{i_2}) \to (f_{i_2},m_{j_2}) \to \dots \to (f_{i_n},m_{j_n}) \to (m_{j_n},f_{i_1}) \to (f_{i_1},m_{j_1})\,.$$
%\end{small}

A similar idea was used by Nicolaas de Bruijn~\cite{deBruijn1946} to construct a shortest sequence, which contains every subsequence of length $\ell$ (called \emph{$\ell$-mer}) over a given alphabet.
He introduced directed graphs, now named after him, whose nodes represent $(\ell-1)$-mers and arcs represent $\ell$-mers 
(the arc corresponding to an $\ell$-mer $s$ connects the prefix of $s$ with the suffix of $s$).

We employ de Bruijn graphs for $\ell=3$ for solving the menage problem. However, in contrast to conventional \emph{unweighted} de Bruijn graphs, 
we introduce algebraic weights to account for (i) the balance between females and males; and (ii) the number of close couples.

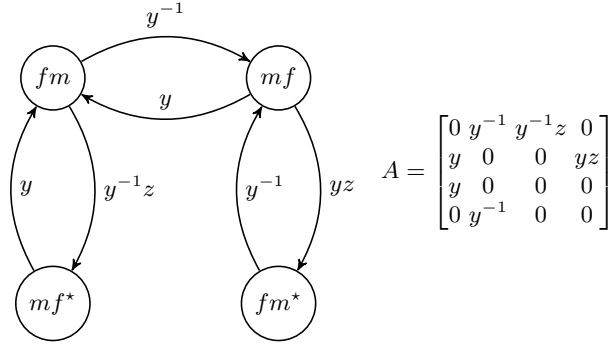
\begin{figure}[t]
\begin{center}
\begin{tabular}{ccc}
\begin{tabular}{c}
\begin{tikzpicture}[->,>=stealth',auto,node distance=3cm,semithick]

  \node[state] (fm)                    {$fm$};
  \node[state] (MF) [below of=fm] {$mf^{\star}$};
  \node[state] (mf) [right of=fm] {$mf$};
  \node[state] (FM) [below of=mf] {$fm^{\star}$};

  \path (fm) edge [bend left] node {$y^{-1}$} (mf)
        (mf) edge [bend left] node[swap] {$y$} (fm)
        (fm) edge [bend left] node {$y^{-1}z$} (MF)
        (mf) edge [bend left] node {$yz$} (FM)
        (MF) edge [bend left] node[swap] {$y$} (fm)
        (FM) edge [bend left] node[swap] {$y^{-1}$} (mf);
\end{tikzpicture}
\end{tabular}
&
\qquad
&
$
A = \begin{bmatrix} 
0 & y^{-1} & y^{-1}z & 0\\
y & 0 & 0 & yz\\
y & 0 & 0 & 0\\
0 & y^{-1} & 0 & 0
\end{bmatrix}
$
\end{tabular}
\end{center}
\caption{The weighted de Bruijn graph for the menage problem and its adjacency matrix $A$.}
\label{fig:deBMen}
\end{figure}

The (weighted) de Bruijn graph for the menage problem and its adjacency matrix $A$ are shown in Fig.~\ref{fig:deBMen}.
This graph has $4$ nodes labeled $fm$ (for clockwise adjacent female--male pair), $mf$ (clockwise adjacent male--female pair), and their starred variants indicating close couples.
There is an arc connecting every pair of nodes $uv$ and $vw$ (at most one of which may be starred) for $u,v,w\in\{f,m\}$.
Each such arc has an algebraic weight $y^pz^q$ with $p=\pm 1$ and $q\in\{0,1\}$ such that the degree of indeterminate $y$ accounts for the males-females balance, 
while the degree of indeterminate $z$ accounts for the number of close couples. Namely, $p=1$ whenever $w=m$ (and $p=-1$ whenever $w=f$), while $q=1$ whenever $vw$ is starred.

Any seating arrangement corresponds to a cyclic sequence of nodes $fm$ and $mf$, some of which may be starred to indicate close couples.
Such sequence with $j$ close couples corresponds to a walk of length $2n$ and algebraic weight $y^0z^j$.
The transfer-matrix method~\cite[Section 4.7]{Stanley1997} implies that the number of such walks equals $[y^0z^j]\ \tr(A^{2n})$, i.e., the coefficient of $y^0z^j$ in the trace of matrix $A^{2n}$.
This leads to a new formula for $W_{n,j}$:
$$
W_{n,j} = [y^0z^j]\ \tr(A^{2n})\cdot j!\cdot (n-j)!^2\,,
$$
where the factors $j!$ and $(n-j)!^2$ bear the same meaning as in \eqref{eq:Wnj}. 
Similarly to \eqref{eq:Mn1}, we then obtain
\beq\label{eq:Mnmatrix}
M_n = n!\cdot \sum_{j=0}^n (-1)^j\cdot (n-j)!\cdot [y^0z^j]\ \tr(A^{2n})\,.
\eeq
Comparison of \eqref{eq:Mnmatrix} and \eqref{eq:Touchard} suggests the following identity, which we will prove explicitly:
\begin{lemma}\label{th:Ay0zj}
For the matrix $A$ defined in Fig.~\ref{fig:deBMen} and any integers $n>1$, $j\geq 0$,
$$[y^0 z^j]\ \tr(A^{2n}) = 2\cdot \frac{2n}{2n-j}\cdot \binom{2n-j}{j}\,.$$
\end{lemma}
\begin{proof}
The eigenvalues of $A$ are $\frac{1\pm\sqrt{1+4z}}{2}$, each of multiplicity 2.\footnote{We remark that $A^2$ does not depend on $y$, 
so it is not surprising that the eigenvalues of $A$ do not depend on $y$ either.}
It follows that
$$[y^0 z^j]\ \tr(A^{2n}) = 2\cdot [z^j]\ \left( \left(\frac{1+\sqrt{1+4z}}{2}\right)^{2n} + \left(\frac{1-\sqrt{1+4z}}{2}\right)^{2n} \right)\,.$$
We further remark that $\frac{1-\sqrt{1+4z}}{2} = -zC(-z)$ and $\frac{1+\sqrt{1+4z}}{2} =C(-z)^{-1}$, 
where $C(x) = \frac{1-\sqrt{1-4x}}{2x}$ is the ordinary generating function for Catalan numbers.

Since $j\leq n$ and $n>1$, we have $[z^j]\ (-zC(-z))^{2n} = 0$. On the other hand, since  $[x^k]\ C(x)^m = \frac{m}{2k+m}\binom{2k+m}{m}$ (e.g., see~\cite[formula (5.70)]{Graham1994}), 
we have
\[
\begin{split}
[z^j]\ C(-z)^{-2n} &= (-1)^j\frac{-2n}{2j-2n}\binom{2j-2n}{j} = \frac{2n}{2n-2j}\binom{2n-j-1}{j} \\
& = \frac{2n}{2n-2j}\binom{2n-j-1}{2n-2j-1} = \frac{2n}{2n-j}\binom{2n-j}{j},
\end{split}
\]
which concludes the proof.
\qed
\end{proof}

Lemma~\ref{th:Ay0zj} proves that our formula \eqref{eq:Mnmatrix} implies the Touchard formula \eqref{eq:Touchard}.
In the next section, we show that it also implies another (apparently new) formula for $M_n$.
But most importantly, the matrix formula~\eqref{eq:Mnmatrix} can be generalized for the ternary menage problem as we show in Section~\ref{sec:DeBter}.

\section{New Formulae for Menage Numbers and Their EGF}

We find it convenient to define the \emph{series Laplace transform} $\L_{x,y}$ of a function $f(x)$ as the conventional Laplace transform of $f(yt)$ (as a function of $t$) evaluated at $1$, i.e.,
$$\L_{x,y}[f] = \int_0^\infty f(yt)\cdot e^{-t}\cdot dt.$$
It can be easily seen that $\L_{x,y}[x^k] = k!\cdot y^k$ for any integer $k\geq 0$ and thus for a power series $f(x)$, we have
$$\L_{x,y}[f] = \sum_{k=0}^{\infty} k!\cdot y^k\cdot [x^k]\ f(x).$$
In particular,
$$\sum_{k=0}^{\infty} k!\cdot [x^k]\ f(x) = \L_{x,1}[f] = \int_0^\infty f(t)\cdot e^{-t}\cdot dt.$$

\begin{lemma}\label{lem:laplace}
Let $U,V$ be same-size square matrices that do not depend on indeterminate $z$. Then for any integer $n\geq 0$,
$$\sum_{j=0}^n (-1)^j\cdot (n-j)!\cdot [z^j]\ \tr((U+V\cdot z)^n) 
= \int_0^\infty \tr((U\cdot t - V)^n)\cdot e^{-t}\cdot dt\,.$$
\end{lemma}
\begin{proof}
We have
$$[z^j]\ \tr((U+V\cdot z)^n) = [z^{n-j}]\ \tr((U\cdot z+V)^n) = (-1)^j\cdot [z^{n-j}]\ \tr((U\cdot z-V)^n)\,.$$
Hence
\[
\begin{split}
\sum_{j=0}^n (-1)^j\cdot (n-j)!\cdot [z^j]\ \tr((U+V\cdot z)^n) &= \sum_{j=0}^n (n-j)!\cdot [z^{n-j}]\ \tr((U\cdot z-V)^n)\\
& = \L_{z,1}[\tr((U\cdot z-V)^n)]\,,
\end{split}
\]
which concludes the proof.
\qed
\end{proof}

We are now ready to derive new closed-form expressions for the menage numbers $M_n$ and their exponential generating function.

\begin{theorem}\label{th:Mn}
For all integers $n>1$,
\beq\label{eq:Mnint}
M_n = 2\cdot n!\cdot \int_0^{\infty}  \left(  \left( \frac{t-2+\sqrt{t^2-4t}}{2} \right)^n + \left( \frac{t-2-\sqrt{t^2-4t}}{2} \right)^n \right)\cdot e^{-t}\cdot dt\,.
\eeq
Furthermore,
\begin{align}
\sum_{n=0}^{\infty} M_n \frac{x^n}{n!} &= -1 - 2x + 2\cdot \int_0^\infty \frac{x^2-1}{xt - (x+1)^2}\cdot e^{-t}\cdot dt\label{eq:gfint}\\
&= -1 + 2x + 2\cdot\frac{1-x}{1+x}\cdot \sum_{k=0}^{\infty} \frac{k!\cdot x^k}{(1+x)^{2k}}  \label{eq:gfser}\\
&= -1 + 2x + 2\cdot\frac{1-x^2}{x}\cdot e^{-\frac{(x+1)^2}x}\cdot \Ei\left(\frac{(x+1)^2}{x}\right)\,,  \label{eq:gfEi} 
\end{align}
where $\Ei(t)$ is the exponential integral.
\end{theorem}

\begin{proof} For the matrix $A$ defined in Fig.~\ref{fig:deBMen}, we have $A^2 = U + V\cdot z$, where 
$$
U = \begin{bmatrix}
 1&0&0&0\\
 0&1&0&0\\
 0&1&0&0\\
 1&0&0&0
 \end{bmatrix}
\quad\text{and}\quad
V = \begin{bmatrix}
 1&0&0&1\\
 0&1&1&0\\
 0&0&1&0\\
 0&0&0&1
 \end{bmatrix}\,.
$$
Then Lemma~\ref{lem:laplace} and formula~\eqref{eq:Mnmatrix} imply %$M_n = n!\cdot \L_{z,1}[\tr((U\cdot z-V)^n)]$.
$$M_n = n!\cdot \int_0^\infty \tr((U\cdot t - V)^n)\cdot e^{-t}\cdot dt\,.$$
Since the eigenvalues of the matrix $U\cdot t - V$ are $\frac{t-2\pm\sqrt{t^2-4t}}{2}$, each of multiplicity 2, we obtain formula \eqref{eq:Mnint}.

To derive \eqref{eq:gfint} from \eqref{eq:Mnint}, we notice that
$$\sum_{n=0}^{\infty} \left( \left( \frac{t-2+\sqrt{t^2-4t}}{2} \right)^n + \left( \frac{t-2-\sqrt{t^2-4t}}{2} \right)^n\right)\cdot x^n = 1+\frac{x^2-1}{xt - (x+1)^2}$$
and take special care of the initial values $M_0=1$ and $M_1=0$.
Expanding the last expression as a power series in $t$, we get
$$\frac{x^2-1}{xt - (x+1)^2} = \frac{1-x}{1+x}\cdot \frac{1}{1-\frac{x}{(x+1)^2}t} = \frac{1-x}{1+x}\cdot \sum_{k=0}^{\infty} \frac{x^k}{(x+1)^{2k}}\cdot t^k.$$
Applying the series Laplace transform $\L_{t,1}$, we obtain \eqref{eq:gfser}.
Expression \eqref{eq:gfEi} now follows from \eqref{eq:gfser}, since $\sum_{k=0}^{\infty} k!\cdot y^k = \frac{e^{-1/y}}{y}\cdot \Ei\left(\frac{1}{y}\right)$ (e.g., see~\cite[formula (1.1.7)]{Bleistein2010}).
\qed
\end{proof}

\section{De Bruijn Graph Approach for Ternary Menage Problem}\label{sec:DeBter}

In contrast to the menage problem, in the ternary case two females or two males can sit next to each other. 
Hence, the de Bruijn graph in this case can be obtained from 
the de Bruijn graph for the menage problem by adding two more nodes labeled $ff$ and $mm$, connected to the other nodes following the same rules (Fig.~\ref{fig:deBTri}).

Let $T_n$ be the number of seating arrangements of $n$ couples in the ternary menage problem.

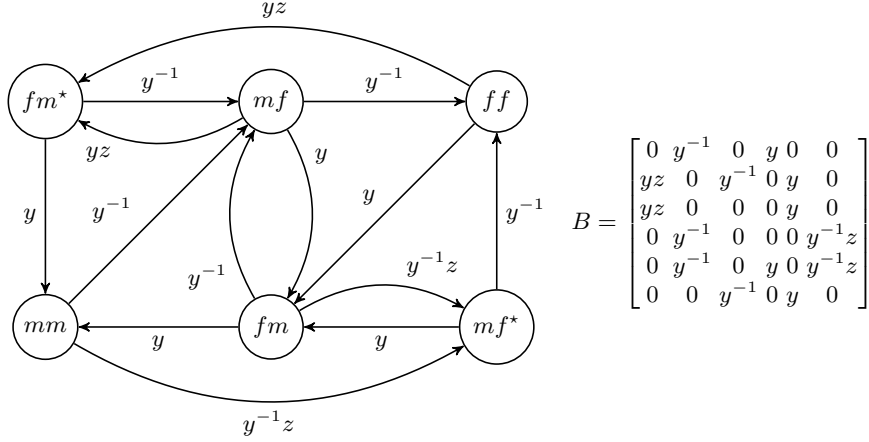
\begin{figure}[t]
\begin{center}
\begin{tabular}{ccc}
\begin{tabular}{c}
\begin{tikzpicture}[->,>=stealth',auto,node distance=3cm,semithick]

  \node[state] (FM)                    {$fm^{\star}$};
  \node[state] (mf) [right of=FM] {$mf$};
  \node[state] (ff) [right of=mf] {$ff$};
  \node[state] (mm) [below of=FM] {$mm$};
  \node[state] (fm) [right of=mm] {$fm$};
  \node[state] (MF) [right of=fm] {$mf^{\star}$};

  \path (fm) edge [bend left] node[near start] {$y^{-1}$} (mf)
        (mf) edge [bend left] node[near start] {$y$} (fm)
        (mf) edge node {$y^{-1}$} (ff)
        (fm) edge node {$y$} (mm)

        (ff) edge [bend right] node[swap] {$yz$} (FM)
        (mm) edge [bend right] node[swap] {$y^{-1}z$} (MF)
	(FM) edge node[swap] {$y$} (mm)
	(FM) edge node {$y^{-1}$} (mf)
        (mf) edge [bend left] node[near end] {$yz$} (FM)

	(MF) edge node[swap] {$y^{-1}$} (ff)
	(MF) edge node [bend left] {$y$} (fm)
        (fm) edge [bend left] node[pos=0.6] {$y^{-1}z$} (MF)

        (mm) edge node[pos=0.4] {$y^{-1}$} (mf)
        (ff) edge node[swap] {$y$} (fm);
\end{tikzpicture}
\end{tabular}
&
\qquad
&
$B= \begin{bmatrix} 
0 & y^{-1} & 0 & y & 0 & 0\\
yz & 0 & y^{-1} & 0 & y & 0\\
yz & 0 & 0 & 0 & y & 0 \\
0 & y^{-1} & 0 & 0 & 0 & y^{-1}z\\
0 & y^{-1} & 0 & y & 0 & y^{-1}z\\
0 & 0 & y^{-1} & 0 & y & 0 
\end{bmatrix}
$
\end{tabular}
\end{center}
\caption{The weighted de Bruijn graph for the ternary menage problem and its adjacency matrix $B$.}
\label{fig:deBTri}
\end{figure}

\begin{theorem}\label{th:Tn}
For $n>1$, the number $T_n$ can be computed in the following ways:
\beq\label{eq:Tnmatrix}
T_n = n!\cdot \sum_{j=0}^n (-1)^j\cdot (n-j)!\cdot [y^0z^j]\ \tr(B^{2n})\,,
\eeq
where $B$ is defined in Fig.~\ref{fig:deBTri}; or
%\begin{small}
\beq\label{eq:TnC}
T_n = n!\cdot\int_0^{\infty} [y^n]\ \tr(B_2^{n})\cdot e^{-t}\cdot dt\,,
~~
%\begin{scriptsize}
B_2 = \begin{bmatrix}
-y& yt& t& 0& yt& -y \\
-y& y(t - 1)& 0& y^2(t - 1)& yt& -y \\
0& y(t - 1)& 0& y^2(t - 1)& 0& -y \\
-y& 0& t - 1& 0& y(t - 1)& 0 \\
-y& yt& t - 1& 0& y(t - 1)& -y \\
-y& yt& 0& y^2t& yt& -y
\end{bmatrix}\,;
%\end{scriptsize}
\eeq
%\end{small}
or
\beq\label{eq:Tnxy}
T_n = n!\cdot\int_0^{\infty} [x^ny^n]\ \frac{a(xy,t) + b(xy,t)\cdot (x+xy^2)}{c(xy,t)+d(xy,t)\cdot (x+xy^2)}\cdot e^{-t}\cdot dt\,,
\eeq
where 
%\begin{small}
\[
\begin{split}
a(p,t) &= -2\,{p}^{5}{t}^{3}+2\,{p}^{4}{t}^{4}+4\,{p}^{5}{t}^{2}-8\,{p}^{4}{t}^{3}-2\,{p}^{5}t+12\,{p}^{4}{t}^{2}-8\,{p}^{4}t+6\,{p}^{3}t-4\,{p}^{2}{t}^{2}\\
&+16\,{p}^{2}-10\,pt+20\,p+6 \,, \\
b(p,t) &= -p^2t(2+p-t)(p-3\,t+6)\,,\\
c(p,t) &= {p}^{6}{t}^{2}-2\,{p}^{5}{t}^{3}+{p}^{4}{t}^{4}+4\,{p}^{5}{t}^{2}-4\,{p}^{4}{t}^{3}-2\,{p}^{5}t+6\,{p}^{4}{t}^{2}-4\,{p}^{4}t+2\,{p}^{3}t-{p}^{2}{t}^{2}\\
&+4\,{p}^{2}-2\,pt+4\,p+1 \,, \\
d(p,t) &= -p^2t(2+p-t)^2\,.
\end{split}
\]
%\end{small}
\end{theorem}

\begin{proof}
Formula~\eqref{eq:Tnmatrix} is similar to \eqref{eq:Mnmatrix} and follows directly from the definition of the de Bruijn graph in Fig.~\ref{fig:deBTri}.

To avoid dealing with negative powers, we notice that $[y^0z^j]\ \tr(B^{2n})= [y^{2n}z^j]\ \tr((yB)^{2n})$. 
Furthermore, the matrix $(yB)^2$ has entries that are polynomial in $y^2$ and $z$ with the degree in $z$ being at most 1,
that is $(yB)^2 = U + V\cdot z$, where matrices $U,V$ do not depend on $z$. Since the specified matrix $B_2$ equals $U\cdot t - V$ where $y^2$ is replaced with $y$,
formula~\eqref{eq:TnC} easily follows from \eqref{eq:Tnmatrix} and Lemma~\ref{lem:laplace}.

According to \cite[Corollary 4.7.3]{Stanley1997}, 
$$\sum_{n=0}^{\infty} \tr(B_2^{n})\cdot x^n = 6-\frac{xQ'(x)}{Q(x)},$$
where $Q(x) = \det(I - x\cdot B_2)$ and $I$ is the $6\times 6$ identity matrix. 
Direct computation shows that $Q(x) = c(xy,t)+d(xy,t)\cdot (x+xy^2)$ and $6Q(x)-xQ'(x)=a(xy,t) + b(xy,t)\cdot (x+xy^2)$, implying that
$$\sum_{n=0}^{\infty} \tr(B_2^{n})\cdot x^n = \frac{a(xy,t) + b(xy,t)\cdot (x+xy^2)}{c(xy,t)+d(xy,t)\cdot (x+xy^2)}\,.$$
Substitution of this expression into \eqref{eq:TnC} yields \eqref{eq:Tnxy}.
\qed
\end{proof}

While formulae~\eqref{eq:Tnmatrix} and \eqref{eq:TnC} provide an efficient way for computing $T_n$ for a given integer $n>1$, 
the special form of the rational function in \eqref{eq:Tnxy} further enables us to obtain a closed-form expression for the EGF for the numbers $T_n$.

\begin{lemma}\label{lem:abcd}
Let $a(z)$, $b(z)$, $c(z)$, $d(z)$ be polynomials such that $[z^0]\ c(z) = 1$ (i.e., $c(z)$ is invertible as a series in $z$).
Then for any integer $n\geq 0$,
%\begin{small}
$$
[x^ny^n]\ \frac{a(xy) + b(xy)\cdot (x+xy^2)}{c(xy) + d(xy)\cdot(x+xy^2)} = [p^n]\ \left(\frac{a(p)\cdot d(p)-b(p)\cdot c(p)}{d(p)\cdot \sqrt{c(p)^2  - 4\cdot p^2\cdot d(p)^2}} + \frac{b(p)}{d(p)}\right) \,.
$$
%\end{small}
\end{lemma}

\begin{proof}
Let $p=xy$. Then $\frac{a(xy) + b(xy)\cdot (x+xy^2)}{c(xy) + d(xy)\cdot(x+xy^2)} = \tfrac{a(p) + b(p)\cdot (x+py)}{c(p) + d(p)\cdot(x+py)}$. 
In the series expansion of this function, we will discard all terms with distinct degrees in $x$ and $y$, while in the remaining terms (with equal degrees in $x$ and $y$), 
we will replace $xy$ with $p$ to obtain a univariate power series in $p$.
We start with the following expansion:
%\begin{small}
\[
\begin{split}
& \frac{a(p) + b(p)\cdot (x+py)}{c(p) + d(p)\cdot(x+py)} 
= \frac{a(p) + b(p)\cdot (x+py)}{c(p)} \sum_{k=0}^{\infty} \left(\frac{-d(p)}{c(p)}\right)^k (x+py)^k \\
& \qquad\qquad\qquad = \frac{a(p)}{c(p)} \sum_{k=0}^{\infty} \left(\frac{-d(p)}{c(p)}\right)^k (x+py)^k + \frac{b(p)}{c(p)} \sum_{k=0}^{\infty} \left(\frac{-d(p)}{c(p)}\right)^k (x+py)^{k+1}\,.
\end{split}
\]
%\end{small}
Here from each power of $x+py$ we extract the term with the equal degrees in $x$ and $y$ and replace it with the corresponding power of $p$. This yields
%\begin{small}
\[
\begin{split}
&\frac{a(p)}{c(p)} \sum_{k=0}^{\infty} \left(\frac{-d(p)}{c(p)}\right)^{2k} \binom{2k}{k} p^{2k} + 
\frac{b(p)}{c(p)} \sum_{k=0}^{\infty} \left(\frac{-d(p)}{c(p)}\right)^{2k+1}  \binom{2k+2}{k+1} p^{2k+2} \\
&= \frac{a(p)}{c(p)}\cdot f\left( \left(\frac{d(p)}{c(p)}p\right)^2\right) - \frac{b(p)}{d(p)}\cdot \left( f\left( \left(\frac{d(p)}{c(p)}p\right)^2\right) - 1\right) \\
&= \frac{a(p)\cdot d(p)-b(p)\cdot c(p)}{c(p)\cdot d(p)}\cdot f\left( \left(\frac{d(p)}{c(p)}\cdot p\right)^2\right) + \frac{b(p)}{d(p)} \\
&= \frac{a(p)\cdot d(p)-b(p)\cdot c(p)}{d(p)\cdot\sqrt{c(p)^2 - 4\cdot p^2\cdot d(p)^2}} + \frac{b(p)}{d(p)}\,,
\end{split}
\]
%\end{small}%
where $f(z) = (1-4z)^{-1/2} = \sum_{k=0}^{\infty} \binom{2k}{k}\cdot z^k$. 
By construction, the coefficient of $x^ny^n$ in the original expression equals the coefficient of $p^n$ in the last expression.
\qed
\end{proof}

\begin{theorem}\label{th:Tegf} 
Let polynomials $a(p,t)$, $b(p,t)$, $c(p,t)$, and $d(p,t)$ be defined as in Theorem~\ref{th:Tn}.
Then for all $n>1$,
\beq\label{eq:Tnpn}
T_n = 
n!\cdot\int_0^{\infty} [p^n]\ \left(\frac{a(p,t)\cdot d(p,t)-b(p,t)\cdot c(p,t)}{d(p,t)\cdot \sqrt{c(p,t)^2  - 4\cdot p^2\cdot d(p,t)^2}} + \frac{b(p,t)}{d(p,t)}\right) \cdot e^{-t}\cdot dt\,.
\eeq
Correspondingly, the exponential generating function for $T_n$ equals
%\begin{scriptsize}
\[
\begin{split}
& \sum_{n=0}^{\infty} T_n \frac{x^n}{n!} = - 2 + 2\cdot x - 2\cdot x\cdot e^{-x-2}\cdot \Ei(x+2) \\
& + 
\int_0^{\infty} 
\frac {\left( {t}^{3}{x}^{2}+ \left( -2\,{x}^{3}-4\,{x}^{2}-x \right) {t}^{2}+ \left( {x}^{4}+4\,{x}^{3}+7\,{x}^{2}+4\,x+3 \right) t-6\,{x}^{2}-9\,x-6 \right)\cdot e^{-t}}{(t-(x+2))
\sqrt{\left( {t}^{2}{x}^{2}-t{x}^{3}-2\,t{x}^{2}-3\,xt+4\,{x}^{2}+4\,x+1 \right)
\left( {t}^{2}{x}^{2}-t{x}^{3}-2\,t{x}^{2}+xt+1 \right)}}\ dt\,.
\end{split}
\]
%\end{scriptsize}
\end{theorem}

\begin{proof}
Formula \eqref{eq:Tnpn} directly follows from \eqref{eq:Tnxy} and Lemma~\ref{lem:abcd}.
Multiplying \eqref{eq:Tnpn} by $\frac{x^n}{n!}$ and summing over $n$, we further get
%\begin{small}
\beq\label{eq:Tgf}
%\begin{split}
\sum_{n=0}^{\infty} T_n \frac{x^n}{n!} = -5 + 2x 
+ \int_0^{\infty} \left(\frac{a(x,t)\cdot d(x,t)-b(x,t)\cdot c(x,t)}{d(x,t)\cdot \sqrt{c(x,t)^2  - 4\cdot p^2\cdot d(x,t)^2}} + \frac{b(x,t)}{d(x,t)}\right) \cdot e^{-t}\cdot dt\,,
%\end{split}
\eeq
%\end{small}%
where the terms $-5 + 2x$ take care of the initial values $T_0=1$ and $T_1=0$. 
While we are not aware if it is possible to simplify the integral of the term involving a square root, below we evaluate the integral of the rational term.

Expansion of $\frac{b(x,t)}{d(x,t)}$ as a power series in $t$ yields
$$\frac{b(x,t)}{d(x,t)} = \frac{x-3t+6}{x-t+2} = 3 - \frac{2x}{2+x}\cdot\frac{1}{1-\frac{t}{2+x}} = 3 - \frac{2x}{2+x}\cdot \sum_{k=0}^{\infty} \frac{t^k}{(2+x)^k}.$$
Similarly to the proof of Theorem~\ref{th:Mn}, this further allows us to evaluate the integral
$$
\int_0^{\infty} \frac{b(x,t)}{d(x,t)} \cdot e^{-t}\cdot dt = 3 - 2\cdot x\cdot e^{-x-2}\cdot \Ei(x+2).
$$
Plugging this expression into \eqref{eq:Tgf} completes the proof.
\qed
\end{proof}

\section{Computing Numerical Values}

The Online Encyclopedia of Integer Sequences~\cite{OEIS} contains a number of sequences related to the menage problem and its ternary variant:

\

\begin{tabular}{|c|l|c|}
\hline
Sequence & Terms for $n=1,2,3,\dots$ & OEIS index\\
\hline\hline
$M_n$ & $0, 0, 12, 96, 3120, 115200, 5836320, 382072320, \dots$  & \texttt{A059375} \\
\hline
$\nicefrac{M_n}{2n!}$ & $0, 0, 1, 2, 13, 80, 579, 4738, 43387, 439792, \dots$ & \texttt{A000179} \\
\hline
$\nicefrac{M_n}{2n}$ & $0, 0, 2, 12, 312, 9600, 416880, 23879520, \dots$ & \texttt{A094047} \\
\hline\hline
$T_n$ & $0, 8, 84, 3456, 219120, 19281600, 2324085120, \dots$ & \texttt{A258338} \\
\hline
$\nicefrac{T_n}{4n}$ & $0, 1, 7, 216, 10956, 803400, 83003040, \dots$  & \texttt{A114939} \\
\hline
\end{tabular}

\

While the Touchard formula \eqref{eq:Touchard} can be used to efficiently compute the menage numbers $M_n$ and associated sequences, 
our formula \eqref{eq:Tnmatrix} enables the same for the numbers $T_n$. In particular, we have computed many terms of sequences \texttt{A114939} and \texttt{A258338} in the OEIS.

We also remark that the formula \eqref{eq:Tnpn} provides another way to compute $T_n$ by extracting the coefficient of $x^p$ (which is a polynomial in $t$) and applying the transform $\L_{t,1}$ 
(i.e., replacing every $t^k$ with $k!$).

\bibliographystyle{splncs03} 
\bibliography{math.bib}

\newcommand{\noopsort}[1]{}
\begin{thebibliography}{10}
\providecommand{\url}[1]{\texttt{#1}}
\providecommand{\urlprefix}{URL }

\bibitem{Bleistein2010}
Bleistein, N., Handelsman, R.A.: Asymptotic expansions of integrals. {Dover
  Books on Mathematics}, {Dover Publications}, revised edn. (2010)

\bibitem{Bogart1986}
Bogart, K.P., Doyle, P.G.: Non-sexist solution of the ménage problem. American
  Mathematical Monthly  93,  514--519 (1986)

\bibitem{deBruijn1946}
{de Bruijn}, N.G.: A combinatorial problem. Proceedings of the Koninklijke
  Nederlandse Akademie van Wetenschappen. Series A  49(7),  758 (1946)

\bibitem{Golin2005}
Golin, M.J., Leung, Y.C.: {Unhooking Circulant Graphs: A Combinatorial Method
  for Counting Spanning Trees and Other Parameters}. In: Hromkovič, J., Nagl,
  M., Westfechtel, B. (eds.) {Graph-Theoretic Concepts in Computer Science},
  Lecture Notes in Computer Science, vol. 3353, pp. 296--307. Springer, Berlin
  Heidelberg (2005)

\bibitem{Graham1994}
Graham, R.L., Knuth, D.E., Patashnik, O.: {Concrete Mathematics: A Foundation
  for Computer Science}. Addison-Wesley, Reading, MA, 2nd edn. (1994)

\bibitem{Lucas1891}
Lucas, E.: Théorie des Nombres. Gauthier-Villars, Paris (1891)

\bibitem{Stanley1997}
Stanley, R.P.: {Enumerative Combinatorics}, vol.~1. Cambridge University Press,
  New York, NY (1997)

\bibitem{OEIS}
{The OEIS Foundation}: {The On-Line Encyclopedia of Integer Sequences}.
  Published electronically at \url{http://oeis.org} (2016)

\bibitem{Touchard1934}
Touchard, J.: {Sur un problème de permutations}. C. R. Acad. Sciences Paris
  198,  631--633 (1934)

\bibitem{Wyman1958}
Wyman, M., Moser, L.: On the probl{\`e}me des m{\'e}nages. Canadian Journal of
  Mathematics  10,  468--480 (1958)

\end{thebibliography}

\end{document}